\documentclass[12pt, reqno]{amsart}
\usepackage{amsmath, amsthm, amscd, amsfonts, amssymb,color,mathrsfs,}

\usepackage[left=2.5cm, right=2.5cm, top=2.5cm, bottom=2.5cm]{geometry}
\usepackage{hyperref}
\hypersetup{
    colorlinks=true,
    linkcolor=blue,
    filecolor=darkmagenta,      
    urlcolor=cyan,
    citecolor=red
} 
\usepackage{amsthm}
\usepackage{graphicx}
\usepackage{tikz-cd} 
\newtheorem{theorem}{Theorem}[section]
\newtheorem{lemma}[theorem]{Lemma}
\newtheorem{proposition}[theorem]{Proposition}
\newtheorem{corollary}[theorem]{Corollary}

\theoremstyle{definition}

\newtheorem{example}[theorem]{Example}
\newtheorem{remark}[theorem]{Remark}

\numberwithin{equation}{section}

\begin{document}

\author{Pankaj Dey and Mithun Mukherjee}

\address{Pankaj Dey, School of Mathematics, Indian Institute of Science Education and Research Thiruvananthapuram, Maruthamala PO, Vithura, Thiruvananthapuram-695551, Kerala, India.}
\email{\textcolor[rgb]{0.00,0.00,0.84}{pankajdey15@iisertvm.ac.in}}

\address{Mithun Mukherjee, School of Mathematical \& Computational Sciences, Indian Association for the Cultivation of Science, 2A \& 2B Raja S C Mullick Road, Jadavpur, Kolkata-700032, India.}
\email{\textcolor[rgb]{0.00,0.00,0.84}{mithun.mukherjee@iacs.res.in}}

\title[Higher Rank Numerical Ranges and Unitary Dilations]{Higher Rank Numerical Ranges and Unitary Dilations}

\begin{abstract}
Here we show that for $k\in \mathbb N,$ the closure of the $k$-rank numerical range of a contraction $A$ acting on an infinite-dimensional Hilbert space $\mathcal{H}$ is the intersection of the closure of the $k$-rank numerical ranges of all unitary dilations of $A$ to $\mathcal{H}\oplus\mathcal{H}.$ The same is true for $k=\infty$ provided the $\infty$-rank numerical range of $A$ is non-empty. These generalize a finite dimensional result of Gau, Li and Wu. We also show that when both defect numbers of a contraction are equal and finite ($=N$), one may restrict the intersection to a smaller family consisting of all unitary $N$-dilations. A result of {Bercovici and Timotin} on unitary $N$-dilations is used to prove it. Finally, we have investigated the same problem for the $C$-numerical range and obtained the answer in negative.
\end{abstract}
\maketitle

\let\thefootnote\relax

\section{Introduction}
Let $\mathcal{H}$ be a Hilbert space and $\mathcal{B(\mathcal{H})}$ be the algebra of all bounded linear maps acting on $\mathcal{H}$. Suppose $A\in \mathcal{B(\mathcal{H})}$. Let $\mathcal{K}$ be a Hilbert space containing $\mathcal{H}$. An operator $B\in\mathcal{B(\mathcal{K})}$ is said to be a \textit{dilation} of $A$ (or, $A$ is said to be a \textit{compression} of $B$) if there exists a projection $P\in\mathcal{B(\mathcal{K})}$ on $\mathcal{H}$ such that $A=P_{\mathcal{H}}B|_{\mathcal{H}}$ or, equivalently, if $B$ is unitarily similar to $2\times 2$ operator matrix $\begin{pmatrix}
    A & *\\
    * & *\\
  \end{pmatrix}$. If $\text{dim }\mathcal{K}\ominus\mathcal{H}=r$ then $B$ is called an \textit{$r$-dilation} of $A$. Moreover, if $B$, being a dilation of $A$, is unitary then $B$ is said to be a \textit{unitary dilation} of $A$. Halmos \cite{PH} showed that every contraction $A\in\mathcal{B(\mathcal{H})}$ has a unitary dilation $U\in\mathcal{B(\mathcal{H}\oplus\mathcal{H})}$ of the form $$U=\begin{pmatrix}
    A & -\sqrt{I-AA^*}\\
    \sqrt{I-A^*A} & A^*\\
  \end{pmatrix}.$$
It generated a lot of research including far reaching {Sz.-Nagy} dilation theorem regarding power unitary dilations of a contraction.

The notion of the quadratic form associated with a matrix has been extended for an operator acting on a Hilbert space, which is known as the numerical range. The \textit{numerical range} of $A\in\mathcal{B\mathcal{(H)}}$, denoted by $W(A)$, is defined as 
\begin{align*}
W(A):=\left\{\langle Ax,x\rangle:\Vert x\Vert=1\right\}.
\end{align*}
It has been studied extensively because of its connections and applications to many different areas. The numerical range is a non-empty, convex set. Durszt  \cite{D} proved that the numerical range of a normal operator $A$ is the intersection of all convex Borel set $s$ such that $E(s)=I$ where $E$ is the unique spectral measure associated with $A$. The closure of the numerical range of a normal operator is the closed convex hull of its spectrum \cite{GR}.

It is, in general, difficult to compute the numerical range. Halmos \cite{PH2} conjectured that for every contraction $A\in\mathcal{B(\mathcal{H})}$,
\begin{align}\label{Halmos conjecture}
W(A)=\bigcap\left\{W(U): U \text{ is a unitary dilation of A}\right\}.
\end{align}
Durszt  \cite{D} settled (\ref{Halmos conjecture}) in negative using the description of the numerical range of a normal operator in terms of the convex Borel sets. Later, {Choi and Li} \cite{CL} proved that for every contraction $A\in\mathcal{B(\mathcal{H})}$,
\begin{align}\label{Choi-Li theorem}
\overline{W(A)}=\bigcap\left\{\overline{W(U)}: U\in\mathcal{B}\mathcal{(H\oplus H)} \text{ is a unitary dilation of A}\right\}.
\end{align}
Note here that the closure sign can be omitted in finite dimensional case.

Let $A\in\mathcal{B}\mathcal{(H)}$ be a contraction. Denote $D_A=(I-A^*A)^\frac{1}{2},\; \mathcal{D}_A=\overline{\text{ran}D_A}$ and $d_A=\text{dim }\mathcal{D}_A$ as the \textit{defect operator}, the \textit{defect space} and the \textit{defect number} of $A$, respectively. {Bercovici and Timotin} \cite{BT} recently showed that if both defect numbers of a contraction are same and finite (=$N$) then one can restrict the intersection to a smaller family consisting of all unitary $N$-dilations. It complements {Choi and Li's} theorem (\ref{Choi-Li theorem}) and also generalizes a result of {Benhida, Gorkin and Timotin} \cite{BGT} for $C_0(N)$ contractions. Let $A\in\mathcal{B(\mathcal{H})}$ be a contraction with $d_A=d_{A^*}=N<\infty$. {BT} proved
\begin{align}\label{Bercovici-Timotin theorem}
\overline{W(A)}=\bigcap\left\{\overline{W(U)}: U \text{ is a unitary } N \text{-dilation of A to }\mathcal{H}\oplus\mathbb{C}^N\right\}.
\end{align}

Choi, Kribs and \.{Z}yczkowski have first defined the higher rank numerical range in the context of ``quantum error correction" \cite{CKZ}, which is defined in the following way. Let $A\in\mathcal{B\mathcal{(H)}}$ and $1\leq k\leq\infty$. The \textit{$k$-rank numerical range} of $A$, denoted by $\Lambda_k(A)$, is defined as
$$\Lambda_k(A):=\left\{\lambda\in\mathbb{C}: PAP=\lambda P,\text{ for some projection } P \text{ of rank }k\right\}$$
or, equivalently, $\lambda\in\Lambda_k(A)$ if and only if there is an orthonormal set $\{f_j\}_{j=1}^k$ such that $\langle Af_j,f_r\rangle=\lambda\delta_{j,r}$ for $j,r\in\{1,2,\cdots,k\}$. Clearly, 
\begin{align*}
W(A)=\Lambda_1(A)\supseteq\Lambda_2(A)\supseteq\cdots\supseteq\Lambda_k(A)\supseteq\cdots.
\end{align*}

Li and Sze \cite{LS} have described the higher rank numerical range of a matrix as an intersection of closed half planes. Let $A\in M_n$ and $1\leq k\leq n$. Li and Sze showed
\begin{align}\label{geometric description}
\Lambda_k(A)=\bigcap\limits_{\theta\in[0,2\pi)}\left\{\mu\in\mathbb{C}:\Re({e^{i\theta}\mu})\leq\lambda_k(\Re({e^{i\theta}A}))\right\},
\end{align}
where $\lambda_k(H)$ denotes the $k$-th largest eigenvalue of the self-adjoint matrix $H\in M_n$. As an immediate consequence, it follows that the higher rank numerical range of a matrix is convex. It also proved that for a normal matrix $A\in M_n$,
\begin{align}\label{HRNR for normal matrix}
\Lambda_k(A)=\bigcap\limits_{1\leq j_1<\cdots<j_{n-k+1}\leq n} \text{ conv }\left\{\lambda_{j_1},\ldots,\lambda_{j_{n-k+1}}\right\}.
\end{align}
However, the convexity of the higher rank numerical range of any operator was shown by Woerdeman \cite{HW}. For the non-emptyness of the higher rank numerical ranges, the reader may refer to \cite{LPS2}, \cite{RM}.

Let $A\in M_n$ be a contraction. Observe that $d_A=d_{A^*}$. {Gau, Li and Wu} \cite{GLW} proved the following which extends and refines (\ref{Choi-Li theorem}) to the higher rank numerical ranges of matrices.

\begin{theorem}[Theorem 1.1, \cite{GLW}]\label{Gau-Li-Wu theorem}
Let $A\in M_n$ be a contraction and $1\leq k\leq n$. Then
\begin{align*}
\Lambda_k(A)=\bigcap\left\{\Lambda_k(U): U\in M_{n+d_A}\text{ is a unitary dilation of A}\right\}.
\end{align*}
\end{theorem}

It is to be noted that there exists a normal contraction $A$ for which the $k$-rank numerical range of all unitary dilations of $A$ contains $\Lambda_k(A)$ as a proper subset for $1\leq k\leq\infty.$ See \cite{DM}. The following is the first main theorem of this paper. It generalizes Theorem \ref{Gau-Li-Wu theorem}.

\begin{theorem}\label{main theorem 1}
Let $A\in\mathcal{B}\mathcal{(H)}$ be a contraction and $k\in\mathbb{N}$. Then
$$\overline{\Lambda_k(A)}=\bigcap\left\{\overline{\Lambda_k(U)}: U\in\mathcal{B}\mathcal{(H\oplus H)} \text{ is a unitary dilation of A}\right\}.$$
Moreover, $$\overline{\Lambda_\infty(A)}=\bigcap\left\{\overline{\Lambda_\infty(U)}: U\in\mathcal{B}\mathcal{(H\oplus H)} \text{ is a unitary dilation of A}\right\}$$ provided $\Lambda_\infty(A)$ is non-empty.
\end{theorem}

An example is given that Theorem \ref{main theorem 1} is not true whenever the $\infty$-rank numerical range is empty (Example \ref{counter example}).

The following is the second main theorem of this paper which complements the previous theorem. It states that if both defect numbers of a contraction are same and finite then one can restrict the intersection to a smaller family of unitay dilations. It generalizes (\ref{Bercovici-Timotin theorem}). A result of {Bercovici and Timotin} on unitary $N$-dilations is used while proving it.

\begin{theorem}\label{main theorem 2}
Let $A\in\mathcal{B}\mathcal{(H)}$ be a contraction with $d_A=d_{A^*}=N<\infty$ and $1\leq k\leq\infty$. Then
$$\overline{\Lambda_k(A)}=\bigcap\left\{\overline{\Lambda_k(U)}: U\text{ is a unitary } N\text{-dilation of A to } \mathcal{H}\oplus\mathbb{C}^N\right\}.$$
\end{theorem}

There are so many other generalizations of the numerical range like the higher rank numerical range. One of them is the $C$-numerical range. Let $A,C\in M_n$. The \textit{$C$-numerical range} of $A$, denoted by $W_C(A)$, is defined as
\begin{align}
W_C(A):=\left\{\text{tr}(CU^*AU): U\in M_n \text{ is a unitary matrix}\right\}.
\end{align}
Let $C$ be unitarily similar to $\textbf{c}=\text{diag}(c_1,c_2,\cdots,c_n)$ where $c_j\in\mathbb{C}$ for all $1\leq j\leq n$. As $W_\textbf{c}(A)$ is unitarily invariant, we obtain
\begin{align}\label{definition of C-numerical range}
W_\textbf{c}(A)=\left\{\sum\limits_{j=1}^n c_j\langle Ae_j,e_j\rangle:\{e_j\}_{j=1}^n \text{ is an orthonormal basis of }\mathbb{C}^n\right\}.
\end{align} 
Note that if $\textbf{c}=\text{diag}(1,0,\cdots,0)$ then $W_\textbf{c}(A)=W(A)$. Westwick \cite{W} has shown that $W_\textbf{c}(A)$ is convex if $c_j\in\mathbb{R}$ for all $1\leq j\leq n$. In addition, he gave an example which shows that for $(c_1,c_2,\cdots,c_n)\in\mathbb{C}^n$ with $n\geq 3$, the $\textbf{c}$-numerical range may fail to be convex. The reader may refer to the survey article \cite{L} for more details on the $C$-numerical range. We have investigated Theorem \ref{main theorem 1}-and Theorem \ref{main theorem 2}-type relations for the $\textbf{c}$-numerical range and have obtained the answer in negative, which is the final theorem of this paper.

\begin{theorem}\label{main theorem 3}
There exist $A, \textbf{c}$ for which the intersection of the closure of the $\textbf{c}\oplus \textbf{0}$-numerical ranges of all unitary dilations of $A$ contains $\overline{W_\textbf{c}(A)}$ as a proper subset.
\end{theorem}

Let us end this section by listing some basic properties of the higher rank numerical range:
\begin{itemize}
\item[(P1)] $\Lambda_k(\alpha A+\beta I)=\alpha\Lambda_k(A)+\beta$,  for $\alpha, \beta\in \mathbb{C}$.
\item[(P2)] $\Lambda_k(A^*)=\overline{\Lambda_k(A)}$.
\item[(P3)] $\Lambda_k(A\oplus B)\supseteq \Lambda_k(A)\cup\Lambda_k(B)$.
\item[(P4)] $\Lambda_k(U^*AU)=\Lambda_k(A)$, for any unitary $U\in\mathcal{B}\mathcal{(H)}$.
\item[(P5)] If $A_0$ is a compression of $A$ on a subspace $\mathcal{H}_0$ of $\mathcal{H}$ such that dim$(\mathcal{H}_0)\geq k$ then $\Lambda_k(A_0)\subseteq\Lambda_k(A)$.
\end{itemize}

\section{Preliminaries}

Throughout this paper, we denote $\mathcal{H}$, an infinite dimensional separable Hilbert space and $\mathcal{B(\mathcal{H})}$, the algebra of all bounded linear maps acting on $\mathcal{H}$. Let $A\in\mathcal{B(\mathcal{H})}$ be a normal operator and $E$ be the unique spectral measure associated with $A$ defined on the Borel $\sigma$-algebra in $\mathbb{C}$ supported on $\sigma(A)$.  Suppose $1\leq k\leq\infty$. Let
\begin{align*}
\mathcal{S}_k:=\left\{H: H\text{ is a half closed-half plane in }\mathbb{C} \text{ with dim ran }E(H)<k \right\}.
\end{align*}
Suppose $V_k(A):=\bigcap\limits_{H\in\mathcal{S}_k} H^c$. Dey and Mukherjee \cite{DM} proved the following, which extends (\ref{HRNR for normal matrix}) for a normal operator acting on an infinite-dimensional Hilbert space.

\begin{theorem}[Theorem 3.2, \cite{DM}]\label{Dey-Mukherjee theorem}
	Let $A\in\mathcal{B(\mathcal{H})}$ be normal and $1\leq k\leq\infty$. Then
	\begin{align*}
	\Lambda_k(A)=\bigcap\limits_{V\in\mathscr{V}_k}W(V^*AV)=V_k(A),
	\end{align*}
	where $\mathscr{V}_k$ is the set of all isometries $V:\mathcal{H}\rightarrow\mathcal{H}$ such that codimension of ran $V$ is less than $k$. 
\end{theorem}

Let $A\in\mathcal{B(\mathcal{H})}$ be a self-adjoint operator and $k\in\mathbb{N}$. Define
\begin{align}
&\lambda_k(A):=\sup\left\{\lambda_k(V^*AV): V:\mathbb{C}^k\rightarrow\mathcal{H}\text{ is an isometry}\right\}.
\end{align}

Let us now observe the following.

\begin{lemma}\label{1}
	Let $A\in\mathcal{B(\mathcal{H})}$ be self-adjoint and $k\in\mathbb{N}$. Suppose $\{\mathcal{H}_n\}_{n=1}^\infty$ is an increasing sequence of closed subspaces of $\mathcal{H}$ with $\mathcal{H}=\overline{\bigcup\limits_{n\geq 1}\mathcal{H}_n}$ and $A_n=P_{\mathcal{H}_n}A|_{\mathcal{H}_n}$. Then 
	\begin{align*}
	\lim\limits_{n\to\infty}\lambda_k(A_n)=\lambda_k(A).
	\end{align*}
\end{lemma}

\begin{proof}
	By interlacing theorem, we have, $\lambda_k(A_n)\leq\lambda_k(A_{n+1})$ for every $n\geq k$. Observe that $\lambda_k(A_n)\leq\lambda_k(A)$ for all $n\geq k$. Indeed, if not then there exists $n_{\circ}\geq k$ such that $\lambda_k(A)<\lambda_k(A_{n_\circ})$. Choose $A_{n_\circ}'=\text{diag}(\lambda_1(A_{n_\circ}),\cdots,\lambda_k(A_{n_\circ}))$. Then $A_{n_\circ}'$ is a $k$-by-$k$ compression of $A$ with $\lambda_k(A)<\lambda_k(A_{n_\circ}')$. It contradicts the definition of $\lambda_k(A)$. So, $\lambda_k(A_n)\leq\lambda_k(A)$ for $n\geq k$. Therefore, $\{\lambda_k(A_n)\}_{n\geq k}$ is a monotonically increasing sequence and bounded above. Hence $\lim\limits_{n\to\infty}\lambda_k(A_n)=\sup\limits_{n\geq k}\lambda_k(A_n)$. We claim that $\sup\limits_{n\geq k}\lambda_k(A_n)=\lambda_k(A)$.

   Given $\epsilon>0$, by the definition of $\lambda_k(A)$, there exists a $k$-by-$k$ compression $A'$ of $A$ such that $\lambda_k(A)-\frac{\epsilon}{2}<\lambda_k(A')$. Note that there exists $n_1\geq k$ and a $k$-by-$k$ compression $A''$ of $A_{n_1}$ such that $\Vert A'-A''\Vert<\frac{\epsilon}{2}$. It implies that $\vert\lambda_k(A')-\lambda_k(A'')\vert\leq\frac{\epsilon}{2}$. So, $\lambda_k(A)-\epsilon<\lambda_k(A'')$. Again, by interlacing theorem, $\lambda_k(A'')\leq\lambda_k(A_{n_1})$. Therefore, $\lambda_k(A)-\epsilon<\lambda_k(A_{n_1})$ for some $n_1\geq k$. Hence $\lim\limits_{n\to\infty}\lambda_k(A_n)=\lambda_k(A)$.
\end{proof}

Let $A\in\mathcal{B\mathcal{(H)}}$ and $k\in\mathbb{N}$. Define
\begin{align*}
&\Omega_k(A):=\bigcap\limits_{\xi\in[0,2\pi)}\left\{\mu\in\mathbb{C}:\Re({e^{i\xi}\mu})\leq\lambda_k(\Re({e^{i\xi}A}))\right\}.
\end{align*}

Let $\text{\textbf{Int}}(S)$ denote the relative interior of $S$ for $S\subseteq\mathbb{C}$. Li, Poon and Sze \cite{LPS} proved the following, which extends (\ref{geometric description}) for an operator.

\begin{theorem}[Theorem 2.1, \cite{LPS}]\label{geomertic description for operator}
	Let $A\in\mathcal{B}\mathcal{(H)}$ and $k\in\mathbb{N}$. Then
	\begin{align*}
	\text{\textbf{Int}}\left(\Omega_k(A)\right)\subseteq\Lambda_k(A)\subseteq\Omega_k(A)=\overline{\Lambda_k(A)}.
	\end{align*}
\end{theorem}

Define $\Omega_{\infty}(A)=\bigcap\limits_{k\geq 1}\Omega_k(A)$.

\begin{theorem}[Theorem 5.2, \cite{LPS}]\label{geomertic description for operator2}
	Let $A\in\mathcal{B}\mathcal{(H)}$. Then
	\begin{align*}
	\text{\textbf{Int}}\left(\Omega_{\infty}(A)\right)\subseteq\Lambda_{\infty}(A)\subseteq\Omega_{\infty}(A).
	\end{align*}
	Moreover, $\overline{\Lambda_{\infty}(A)}=\Omega_{\infty}(A)$ if and only if $\Lambda_\infty(A)\neq\emptyset.$
\end{theorem}

Let $r>0$. Denote $D(0,r):=\{z\in\mathbb{C}:\vert z\vert<r\}$, the disc centred at $0$ with radius $r$ and $\mathbb{D}=D(0,1)$, the open unit disc. Let us end this section with the following lemmas and corollary.

\begin{lemma}[Theorem 2.1, \cite{HG}]\label{k-rank numerical range of finite dimensional shift}
	Suppose $n\geq 2$ and $1\leq k\leq n$. Let $S_n$ be the $n$-dimensional unilateral shift on $\mathbb{C}^n$. Then
	\begin{align*}
	\Lambda_k(S_n)=\begin{cases}
	\overline{D(0,\cos\frac{k\pi}{n+1})},  & \text{if $1\leq k\leq[\frac{n+1}{2}]$} \\
	\emptyset, & \text{if $[\frac{n+1}{2}]<k\leq n.$}
	\end{cases}
	\end{align*}
\end{lemma}

\begin{lemma}\label{k-rank numerical range of shift}
	Let $S$ be a shift operator with any multiplicity acting on an infinite dimensional Hilbert space. Then $\Lambda_k(S)=\mathbb{D}$ for all $k\in\mathbb{N}$.
\end{lemma}

\begin{proof}
	Without loss of generality, we may consider $S:l^2(\mathcal{H})\rightarrow l^2(\mathcal{H})$ such that 
	\begin{align*}
	S\left((x_0,x_1,...)\right)=(0,x_0,x_1,...),\qquad (x_0,x_1,...)\in l^2(\mathcal{H}),
	\end{align*}
	where $\mathcal{H}$ is a Hilbert space of dimension same as the multiplicity of $S$. So, $S=S_0\otimes I_\mathcal{H}$, where $S_0$ is the unilateral shift on $l^2(\mathbb{N})$. Let $S_n$ be the unilateral shift on $\mathbb{C}^n$. As $S_n$ is a compression of $S_0$, by Lemma \ref{k-rank numerical range of finite dimensional shift} and (P5), we obtain $\overline{D(0,\cos(\frac{k\pi}{n+1}))}=\Lambda_k(S_n)\subseteq\Lambda_k(S_0)$ for $1\leq k\leq\left[\frac{n+1}{2}\right]$. Taking $n\rightarrow \infty$, we get $\mathbb{D}\subseteq\Lambda_k(S_0)$. Let us now show that $\Lambda_k(S_0)\subseteq\Lambda_k(S)$. Let $\lambda\in\Lambda_k(S_0)$. Then there exists an isometry $X:\mathbb{C}^k\rightarrow l^2(\mathbb{N})$ such that $X^*S_0X=\lambda I_k$. Let $h_0\in\mathcal{H}$ with $\Vert h_\circ\Vert=1$. Define $Y:\mathbb{C}^k\rightarrow l^2(\mathbb{N})\otimes \mathcal{H}\approx l^2(\mathcal{H})$ such that $Yh=Xh\otimes h_0, \; h\in\mathbb{C}^k$. Let $h\in\mathbb{C}^k$ with $\Vert h\Vert=1$. Now, 
	\begin{align*}
	\Vert Yh\Vert^2=\langle Xh\otimes h_0,Xh\otimes h_0\rangle=\langle Xh,Xh\rangle\langle h_0,h_0\rangle=\Vert Xh\Vert^2\Vert h_0\Vert^2=1.
	\end{align*}
	So, $Y$ is an isometry. We claim that $Y^*SY=\lambda I_k$. Let $h\in\mathbb{C}^k$ with $\Vert h\Vert=1$. Observe,
	\begin{align*}
	\langle Y^*SY h,h\rangle=\langle (S_0\otimes I_K)Xh\otimes h_0,Xh\otimes h_0\rangle=\langle X^*S_0Xh,h\rangle=\lambda.
	\end{align*}
	So, $Y^*SY=\lambda I_k$. Therefore, $\lambda\in\Lambda_k(S)$. Hence $\mathbb{D}\subseteq\Lambda_k(S_0)\subseteq\Lambda_k(S)$.

	Let $\lambda\in\Lambda_k(S)$ with $\vert\lambda\vert=1$. Then there exists an orthonormal set $\{f_j\}_{j=1}^k$ such that $\langle Sf_j,f_r\rangle=\lambda\delta_{j,r}$ for $1\leq j,r\leq k$. Note,
	\begin{align*}
	1=\vert\lambda\vert=\vert\langle Sf_j,f_j\rangle\vert\leq\Vert Sf_j\Vert\Vert f_j\Vert=1
	\end{align*}
	for all $1\leq j\leq k$. As Cauchy-Schwarz inequality is being attained for all $1\leq j\leq k$, $\lambda$ is an eigenvalue of $S$ with multiplicity $k$. It contradicts that $\sigma_p(S)=\emptyset$. Hence $\Lambda_k(S)=\mathbb{D}$.
\end{proof}

\begin{corollary}\label{properisometry}
Suppose $V$ is a proper isometry. Then $\overline{\Lambda_k(V)}=\overline{\mathbb{D}}.$
\end{corollary}

\begin{proof}
As $V$ is a proper isometry, by Wold decomposition (c.f. \cite{NF}), we can write $V=V_0\oplus S$, where $V_0$ is unitary and $S$ is a shift operator.  So, by (P3) and Lemma \ref{k-rank numerical range of shift}, we have,
\begin{align}\label{c}
\overline{\mathbb{D}}\supseteq W(V)\supseteq\cdots\supseteq\Lambda_k(V)=\Lambda_k(V_0\oplus S)\supseteq\Lambda_k(V_0)\cup\Lambda_k(S)\supseteq\Lambda_k(S)=\mathbb{D}. 
\end{align}
Taking closure in both sides of (\ref{c}), we get $\overline{\Lambda_k(V)}=\overline{\mathbb{D}}$.
\end{proof}

\section{Proof of Theorem \ref{main theorem 1}}

Let $A\in M_n$ and $U=\begin{pmatrix}
    A & C\\
    B & D\\
  \end{pmatrix}\in M_{2n}$ be a unitary dilation of $A$. By polar decomposition of a matrix and using the fact that $U$ is unitary, we obtain $B=U_1\sqrt{I-A^*A},\; C=-\sqrt{I-AA^*}U_2$ and $D=U_1A^*U_2$, where $U_1,U_2\in M_n$ are unitary. Then
\begin{align*}
(I\oplus U_1^*)U(I\oplus U_1)=\begin{pmatrix}
    A & -\sqrt{I-A^*A}U_{\circ}\\
    \sqrt{I-A^*A} & A^*U_{\circ}\\
  \end{pmatrix},
\end{align*}
where $U_{\circ}=U_2U_1$. Hence we may take any unitay dilation $U$ of $A$ in the form
\begin{align}\label{3}
U=\begin{pmatrix}
    A & *\\
    \sqrt{I-A^*A} & *\\
  \end{pmatrix}.
\end{align}

Let us begin with the following proposition.

\begin{proposition}\label{contractive dilation}
Let $A,B\in\mathcal{B}(\mathcal{H)}$ with $A^*A+B^*B\leq I_{\mathcal{H}}$ and $k\in\mathbb{N}$. Then there exist $C,D\in \mathcal{B}(\mathcal{H)}$ such that $$Z=\begin{pmatrix}
    A & C\\
    B & D\\
  \end{pmatrix}\in \mathcal{B}(\mathcal{H\oplus H)} $$ is a contractive dilation of A with $\lambda_k(Z+Z^*)=\lambda_k(A+A^*)$.
\end{proposition}

\begin{proof}

Let us first assume that dim($\mathcal{H})=n<\infty$. Then, by Theorem 1.1, \cite{GLW}, there exists a unitary dilation $U_0\in M_{n+d_{A}}$ of $A$ such that $\lambda_k(U_0+U_0^*)=\lambda_k(A+A^*)$. Take $U=U_0\oplus(-I)\in M_{2n}$. Then $U$ is a unitary dilation of $A$ with $\lambda_k(U+U^*)=\lambda_k(U_0+U_0^*)=\lambda_k(A+A^*)$. In view of (\ref{3}), we may take $U$ in the form
$$U=\begin{pmatrix}
A & *\\
\sqrt{I-A^*A} & *\\
\end{pmatrix}.$$
Since $A^*A+B^*B\leq I_{\mathcal{H}}$, we have $B^*B\leq C^*C$, where $C=\sqrt{I-A^*A}$. So, $B=J\sqrt{I-A^*A}$ for some contraction $J\in M_n$. Let $$V=\begin{pmatrix}
I_n & 0_n\\
0_n & J^*\\
0_n & -\sqrt{I_n-JJ^*}
\end{pmatrix}.$$
Then $V^*V=I_{2n}$. So, $V$ is an isometry. Let $\widetilde{U}=U\oplus(-I_n)$. Take $Z=V^*\widetilde{U}V\in M_{2n}$. Then $Z$ is a contractive dilation of $A$ with the desired form. By Cauchy's interlacing theorem (Corollary III.1.5, \cite{RB}), we  have $\lambda_k(Z+Z^*)\geq \lambda_k(A+A^*)$ and $\lambda_k(\widetilde{U}+\widetilde{U}^*)\geq \lambda_k(Z+Z^*)$ as $Z$ is a dilation of $A$ and $\widetilde{U}$ is a dilation of $Z$. So,
\begin{align*}
\lambda_k(Z+Z^*) 
&\leq\lambda_k(\widetilde{U}+\widetilde{U}^*)\\
&=\lambda_k(U+U^*)\\
&=\lambda_k(A+A^*)\\
&\leq \lambda_k(Z+Z^*).
\end{align*} 
It implies that $\lambda_k(Z+Z^*)=\lambda_k(A+A^*)$. This completes the proof of the proposition whenever $\mathcal{H}$ is of dimension $n<\infty$.

Now, let $\mathcal{H}$ be an infinite-dimensional separable Hilbert space with an orthonormal basis $\{e_1,e_2,\cdots\}$. Take $A=(a_{ij})_{1\leq i,j\leq \infty}$ and $B=(b_{ij})_{1\leq i,j\leq \infty}$ with respect to the orthonormal basis $\{e_1,e_2,\cdots\}$. Let $A_n=(a_{ij})_{1\leq i,j\leq n}$ and $B_n=(b_{ij})_{1\leq i,j\leq n}$ be the finite sections of $A$ and $B$ respectively. As $A^*A+B^*B\leq I_{\mathcal{H}}$, we have $A_n^*A_n+B_n^*B_n\leq I_{\mathcal{H}}$. So, by the above finite dimensional result, there exists a contractive dilation $$Z_n=\begin{pmatrix}
A_n & C_n\\
B_n & D_n\\
\end{pmatrix}\in M_{2n} $$
of $A_n$ with $\lambda_k(Z_n+Z_n^*)=\lambda_k(A_n+A_n^*)$. Consider
$$\widetilde{Z}_n=
\begin{pmatrix}
\begin{array}{cc|cc}
A_n & \textbf{0} & C_n &\textbf{ 0} \\
\textbf{0} & \textbf{0 }& \textbf{0} & \textbf{0} \\
\hline
B_n & \textbf{0} & D_n & \textbf{0}\\
\textbf{0} &\textbf{ 0} & \textbf{0} & \textbf{0}
\end{array}
\end{pmatrix}\in \mathcal{B}(\mathcal{H\oplus H)}.
$$
Note that $\widetilde{Z}_n$ converges in weak operator topology to
\begin{align*}
Z=\begin{pmatrix}
A & C\\
B & D\\
\end{pmatrix}.
\end{align*}
Now, by applying Lemma \ref{1}, we obtain $\lim\limits_{n\rightarrow \infty}\lambda_k(A_n+A_n^*)=\lambda_k(A+A^*)$ and $\lim\limits_{n\rightarrow \infty}\lambda_k(Z_n+Z_n^*)=\lambda_k(Z+Z^*)$. Finally, since $\lambda_k(Z_n+Z_n^*)=\lambda_k(A_n+A_n^*)$ for all $n\geq k$, we have $\lambda_k(Z+Z^*)=\lambda_k(A+A^*)$.
This completes the proof. 
\end{proof}

The following theorem plays the key role while proving Theorem \ref{main theorem 1}.

\begin{theorem}\label{key theorem}
Let $A\in\mathcal{B\mathcal{(H)}}$ be a contraction and $k\in\mathbb{N}$. Then there exists a unitary dilation $U\in \mathcal{B\mathcal{(H\oplus H)}}$ of $A$ such that $\lambda_k(U+U^*)=\lambda_k(A+A^*).$ 
\end{theorem}

\begin{proof}

If $\lambda_k(A+A^*)=2$ then the following unitary dilation $$U=\begin{pmatrix}
A & -\sqrt{I-AA^*}\\
\sqrt{I-A^*A} & A^*\\
\end{pmatrix}\in \mathcal{B\mathcal{(H\oplus H)}}$$ of $A$  will do the job. Therefore, assume that $\lambda_k(A+A^*)=\mu<2$.

Take $B=\sqrt{I-A^*A}\in\mathcal{B\mathcal{(H)}}$. Then, by Proposition \ref{contractive dilation}, there exists a contractive dilation $$Z_1=\begin{pmatrix}
    A & C\\
    \sqrt{I-A^*A} & D\\
  \end{pmatrix}\in \mathcal{B}(\mathcal{H\oplus H)} $$ 
  of $A$ such that $\lambda_k(Z_1+Z_1^*)=\lambda_k(A+A^*)$ and $\Vert Z_1v\Vert =\Vert v \Vert$ for all $v\in\mathcal{H\oplus O}$, where $\mathcal{O}$ is the zero subspace of $\mathcal{H}$.

Repeating the argument on $Z_1$, we get a contractive dilation $$Z_2=\begin{pmatrix}
    Z_1 & \widetilde{C}\\
    \sqrt{I-Z_1^*Z_1} & \widetilde{D}\\
  \end{pmatrix}\in \mathcal{B}(\mathcal{H\oplus H\oplus H\oplus H)} $$ 
  of $Z_1$ such that $\lambda_k(Z_2+Z_2^*)=\lambda_k(Z_1+Z_1^*)=\lambda_k(A+A^*)$ and $\Vert Z_2v\Vert =\Vert v \Vert$ for all $v\in\mathcal{H\oplus H\oplus O\oplus O}$. Continuing this process, we obtain a contractive dilation $Z_\infty$, denoted by $U$, acting on $\mathcal{H\oplus H\oplus H\oplus\cdots}$ such that $\lambda_k(U+U^*)=\lambda_k(A+A^*)$ and $\Vert Uv\Vert =\Vert v \Vert$ for all unit vector $v\in \mathcal{H\oplus H\oplus H\oplus\cdots}$. Identifying $\mathcal{O\oplus H}$ with $\mathcal{O\oplus H\oplus H\cdots}$, we  may regard $U$ as an isometry acting on $\mathcal{H}\oplus\mathcal{H}$ while $A$ acts on $\mathcal{H}\oplus\mathcal{O}$. Hence we get an isometric dilation $U\in\mathcal{B\mathcal{(H\oplus H)}}$ of $A$ such that $\lambda_k(U+U^*)=\lambda_k(A+A^*)=\mu<2.$

We will now show that $U$ is unitary. If possible, assume that $U$ is a proper isometry. Then, by Corollary \ref{properisometry}, we get $\overline{\Lambda_k(U)}=\overline{\mathbb{D}}$. This forces $\overline{\Lambda_k(U+U^*)}=[-2,2].$ This  contradicts our assumption that $\lambda_k(U+U^*)=\mu<2.$ This completes the proof.  
\end{proof}

\begin{proof}[\textbf{Proof of Theorem \ref{main theorem 1}}]
	Suppose $1\leq k\leq\infty$. Clearly,
	\begin{align*}
	\overline{\Lambda_k(A)}&\subseteq\bigcap\left\{\overline{\Lambda_k(U)}: U\in\mathcal{B}\mathcal{(H\oplus H)} \text{ is a unitary dilation of } A\right\}.
	\end{align*}

	To prove the reverse inclusion, let us first assume that $k\in\mathbb{N}$. Suppose $\xi\notin\overline{\Lambda_k(A)}$. Then, by Theorem \ref{geomertic description for operator}, there exists $\theta\in[0,2\pi)$ such that $e^{i\theta}\xi+e^{-i\theta}\overline{\xi}>\lambda_k(e^{i\theta}A+e^{-i\theta}A^*)$. Now, by Theorem \ref{key theorem}, there exists a unitary dilation $U\in \mathcal{B\mathcal{(H\oplus H)}}$ of $A$ such that $\lambda_k(e^{i\theta}U+e^{-i\theta}U^*)=\lambda_k(e^{i\theta}A+e^{-i\theta}A^*)$. So, $e^{i\theta}\xi+e^{-i\theta}\overline{\xi}>\lambda_k(e^{i\theta}U+e^{-i\theta}U^*)$. Therefore, again by Theorem \ref{geomertic description for operator}, we have $\xi\notin\overline{\Lambda_k(U)}$.

	Let $\xi\notin\overline{\Lambda_{\infty}(A)}$. Then, by Theorem \ref{geomertic description for operator2}, there exists $k_o\in\mathbb{N}$ such that $\xi\notin\Omega_{k_o}(A)$. So, by Theorem \ref{geomertic description for operator}, there exists $\theta\in [0,2\pi)$ such that $e^{i\theta}\xi+e^{-i\theta}\bar{\xi}>\lambda_{k_o}(e^{i\theta}A+e^{-i\theta}A^*)$. Now, by Theorem \ref{key theorem}, there exists a unitary dilation $U\in\mathcal{B}\mathcal{(H\oplus H)}$ of $A$ such that $\lambda_{k_o}(e^{i\theta}A+e^{-i\theta}A^*)=\lambda_{k_o}(e^{i\theta}U+e^{-i\theta}U^*)$. So, $e^{i\theta}\xi+e^{-i\theta}\bar{\xi}>\lambda_{k_o}(e^{i\theta}U+e^{-i\theta}U^*)$. Therefore, by Theorem \ref{geomertic description for operator}, $\xi\notin\Omega_{k_0}(U)$. Hence, by Theorem \ref{geomertic description for operator2}, $\xi\notin\bigcap\limits_{k\geqslant 1}\Omega_k(U)=\Omega_{\infty}(U)=\overline{\Lambda_{\infty}(U)}$ as $\emptyset\neq\Lambda_{\infty}(A)\subseteq\Lambda_{\infty}(U)$. This completes the proof.
\end{proof}

The following example shows that Theorem \ref{main theorem 1} is not true whenever the $\infty$-rank numerical range is empty.

\begin{example}\label{counter example}
	Consider $A=\bigoplus\limits_{n\geq 2}\begin{pmatrix}
	-\frac{1}{n} & 0\\
	0 & \frac{e^\frac{i\pi}{n}}{n}\\
	\end{pmatrix}$. Then
	\begin{align*}
	&\sigma_e(A)=\{-\frac{1}{n}:n\geq 2\}\cup\{\frac{e^\frac{i\pi}{n}}{n}:n\geq 2\},\\
	&\sigma(A)=\sigma_e(A)\cup\{0\}.
	\end{align*}
	Using Theorem \ref{Dey-Mukherjee theorem}, it was shown in Example 5.2, \cite{DM} that $\Lambda_\infty(A)=\emptyset$. Let $U$ be a unitary dilation of $A$. If possible, assume that $0\notin\Lambda_\infty(U)=\bigcap\limits_{k\geq 1}\Lambda_k(U)$. Then there exists $k\in\mathbb{N}$ such that $0\notin\Lambda_k(U)$. By Theorem \ref{Dey-Mukherjee theorem}, there exists a half closed-half plane $H_\circ$ at $0$ such that $\text{dim ran}E_{U}(H_\circ)<k$. We claim that $H_\circ$ cannot contain infinitely many eigenvalues of $A$. If possible, let $H_\circ$ contain infinitely many eigenvalues of $A$, say, $\{\lambda_r:r\geq 1\}$. Take $A'=\bigoplus\limits_{r\geq 1}(\lambda_r)$. Then $$\Lambda_k(A')\subseteq\Lambda_k(A)\subseteq\Lambda_k(U)\subseteq H_\circ^c.$$ It is a contradiction as $\emptyset\neq\Lambda_k(A')\subseteq\text{conv}\{\lambda_r:r\geq 1\}\subseteq H_\circ$. So, the only possible choice of $H_\circ=\{z\in\mathbb{C}:\Im{(z)}<0\}\cup[0,\infty)$. Let $H_{-\frac{1}{2}}=\{z\in\mathbb{C}:\Im{(z)}<0\}\cup[-\frac{1}{2},\infty)$ be a half closed-half plane at $-\frac{1}{2}$. As $\text{dim ran}E_{U}(H_\circ)<k$, we have $\text{dim ran}E_U{(H_{-\frac{1}{2}}})<k$. Consider $A''=\bigoplus\limits_{n\geq 2}(-\frac{1}{n})$. Then $$\Lambda_k(A'')\subseteq\Lambda_k(A)\subseteq\Lambda_k(U)\subseteq H_{-\frac{1}{2}}^c.$$ It is again a contradiction as $\emptyset\neq\Lambda_k(A'')\subseteq\text{conv}\{-\frac{1}{n}:n\geq 2\}\subseteq H_{-\frac{1}{2}}$. Hence $0\in\Lambda_\infty(U)$ for every unitary dilation $U$ of $A$. This provides an example that Theorem \ref{main theorem 1} is not true whenever the $\infty$-rank numerical range is empty.
\end{example}

\begin{remark}
	The contraction $A$, in Example \ref{counter example}, shows that if $W(A)$ lies in a half closed-half plane, it does not necessarily imply that there exists a unitary dilation $U$ of $A$ with $W(U)$ lying in the same half closed-half plane. 
\end{remark}

\section{Proof of Theorem \ref{main theorem 2}}

We begin with a few lemmas.

\begin{lemma}[Proposition 2.2, \cite{BT}]\label{Berkovici and Timotin's unitary N-dilation}
Let $A\in\mathcal{B}(\mathcal{H})$ be a contraction with $d_A=d_{A^*}=N<\infty$. Assume that $\lambda_1,\cdots,\lambda_r$ are distinct points in $\mathbb{T}\setminus\sigma(T)$ and $n_1,\cdots,n_r$ are positive integers satisfying $\sum\limits_{j=1}^r n_j=N$. Then there exists a unitary $N$-dilation $U$ of $T$ such that $\lambda_j$ is an eigenvalue of $U$ with multiplicity greater than or equal to $n_j$ for every $j\in\{1,2,\cdots,r\}$.
\end{lemma}

Let $A\in\mathcal{B(\mathcal{H})}$ be self-adjoint and $k\in\mathbb{N}$. Define
\begin{align}
	\mu_k(A)=\inf\limits_{\substack{\mathcal{N}\subseteq\mathcal{H} \\ codim\mathcal{N}<k}}\sup\limits_{\substack{x\in \mathcal{N} \\ \Vert x\Vert=1}}\langle Ax,x\rangle.
\end{align}
We claim that $\lambda_k(A)\leq\mu_k(A)$. Indeed, if $\mathcal{M,N}$ are two closed subspaces of $\mathcal{H}$ with $\text{dim}\mathcal{M}=k$ and $\text{codim}\mathcal{N}<k$ then there exists a unit vector $h\in\mathcal{M}\cap\mathcal{N}$. Now,
\begin{align}\label{2}
	&\min\limits_{\substack{x\in \mathcal{M} \\ \Vert x\Vert=1}}\langle Ax,x\rangle\leq\langle Ah,h\rangle\leq\sup\limits_{\substack{x\in \mathcal{N} \\ \Vert x\Vert=1}}\langle Ax,x\rangle\nonumber\\
	&\Rightarrow\lambda_k(A)=\sup\limits_{\substack{\mathcal{M}\subseteq \mathcal{H}\\ dim \mathcal{M}=k}}\min\limits_{\substack{x\in \mathcal{M} \\ \Vert x\Vert=1}}\left\langle Ax,x\right\rangle\leq\sup\limits_{\substack{x\in \mathcal{N} \\ \Vert x\Vert=1}}\langle Ax,x\rangle\nonumber\\
	&\Rightarrow\lambda_k(A)\leq\inf\limits_{\substack{\mathcal{N}\subseteq\mathcal{H} \\ codim\mathcal{N}<k}}\sup\limits_{\substack{x\in \mathcal{N} \\ \Vert x\Vert=1}}\langle Ax,x\rangle=\mu_k(A).
\end{align}

\begin{lemma}\label{key lemma for proving main theorem 2}
Let $A\in\mathcal{B}(\mathcal{H})$ be a contraction with $d_A=d_{A^*}=N<\infty$ and $k\in\mathbb{N}$. Then there exists a unitary $N$-dilation $U$ of $A$ such that $\lambda_k(U+U^*)=\lambda_k(A+A^*)$.
\end{lemma}
\begin{proof}
Let $\lambda_k(\Re{(A)})=\mu$. If $\mu=1$ then any unitary $N$-dilation of $A$ will do the job. So, assume that $-1\leq\mu<1$. Let $\epsilon>0$ be such that the line passing through $\mu+\epsilon$ and parallel to the $Y$-axis cuts the unit circle at two points, say, $\lambda_{\epsilon}$ and $\overline{\lambda_\epsilon}$. Then, by Lemma \ref{Berkovici and Timotin's unitary N-dilation}, there exists a unitary $N$-dilation $U_{\epsilon}$ (acting on $\mathcal{H}\oplus\mathbb{C}^N$) of $A$ such that $\lambda_{\epsilon}$ is an eigenvalue of $U_\epsilon$ with multiplicity $N$. Let $E_\epsilon$ be the spectral measure associated with $\Re{(U_\epsilon)}$. We claim that $\text{dim ran}E_\epsilon(\mu+\epsilon,\infty)<k$. If possible, let $\text{dim ran}E_\epsilon(\mu+\epsilon,\infty)\geq k$. Suppose $f_1^\epsilon,\cdots,f_N^\epsilon$ are orthonormal eigenvectors of $\Re{(U_\epsilon)}$ corresponding to the eigenvalue $\mu+\epsilon$ and $\{f_{N+1}^\epsilon,\cdots,f_{N+r}^\epsilon\}$ is an orthonormal basis of $\text{ran}E_\epsilon(\mu+\epsilon,\infty)$. Consider $\mathcal{K}_\epsilon=\overline{\text{span}}\{f_1^\epsilon,f_2^\epsilon,\cdots,f_{N+r}^\epsilon\}\cap\mathcal{H}$. As codimension of $\mathcal{H}$ in $\mathcal{H}\oplus\mathbb{C}^N$ is $N$, we have $\text{dim}(\mathcal{K}_\epsilon)\geq r\geq k$. Note, $P_{\mathcal{K}_\epsilon}\Re{(U_\epsilon)}|_{\mathcal{K}_\epsilon}=P_{\mathcal{K}_\epsilon}\Re{(A)}|_{\mathcal{K}_\epsilon}$. Let $A'$ be any $k$-by-$k$ compression of $P_{\mathcal{K}_\epsilon}\Re{(A)}|_{\mathcal{K}_\epsilon}$. Then
\begin{align*}
\left[\lambda_k(A'),\lambda_1(A')\right]=W(A')\subseteq W(P_{\mathcal{K}_\epsilon}\Re{(A)}|_{\mathcal{K}_\epsilon})=W(P_{\mathcal{K}_\epsilon}\Re{(U_\epsilon)}|_{\mathcal{K}_\epsilon})\subseteq [\mu+\epsilon,\infty).
\end{align*}
So, $\lambda_k(A')\geq\mu+\epsilon>\mu$. It contradicts $\lambda_k(\Re{(A)})=\mu$ as, by the definition of $\lambda_k(\Re{(A)})$, there cannot exist any $k$-by-$k$ compression of $\Re{(A)}$ whose smallest eigenvalue is stricly greater than $\lambda_k(\Re{(A)})=\mu$. Hence $\text{dim ran}E_\epsilon(\mu+\epsilon,\infty)<k$. Now, using (\ref{2}), we have
\begin{align}\label{4}
\lambda_k(\Re{(U_\epsilon}))
&\leq\mu_k(\Re{(U_\epsilon}))\nonumber\\
&=\inf\limits_{\substack{\mathcal{M}\subseteq\mathcal{H} \\ codim\mathcal{M}<k}}\sup\limits_{\substack{x\in \mathcal{M} \\ \Vert x\Vert=1}}\langle \Re{(U_\epsilon)}x,x\rangle\nonumber\\
&\leq\sup\limits_{\substack{x\in\mathcal{N}\\ \Vert x\Vert=1}}\langle \Re{(U_\epsilon)}x,x\rangle,\text{ where }\mathcal{N}=\text{ran}E_\epsilon(-\infty,\mu+\epsilon]\nonumber\\
&\leq\mu+\epsilon.
\end{align}
Since the set of all unitary $N$-dilations of $A$ on $\mathcal{H}\oplus\mathbb{C}^N$ is compact with respect to the norm topology,  $\{U_\epsilon\}_{\epsilon}$ has a limit point, say, $U$. Clearly, $U$ is a unitary $N$-dilation of $A$. Now, using Corollary III.1.2, \cite{RB}, we obtain
\begin{align*}
\lambda_k(\Re{(U)})
&=\sup\limits_{\substack{\mathcal{M}\subseteq \mathcal{H}\oplus\mathbb{C}^N \\ dim\mathcal{M}=k}}\min\limits_{\substack{x\in \mathcal{M} \\ \Vert x\Vert=1}}\langle\Re{(U)}x,x\rangle\\
&=\lim\limits_{\epsilon\rightarrow 0}\sup\limits_{\substack{\mathcal{M}\subseteq \mathcal{H}\oplus\mathbb{C}^N \\ dim\mathcal{M}=k}}\min\limits_{\substack{x\in \mathcal{M} \\ \Vert x\Vert=1}}\langle\Re{(U_\epsilon)}x,x\rangle,\text{ as } U_\epsilon\rightarrow U \text{ in norm}\\
&=\lim\limits_{\epsilon\rightarrow 0}\lambda_k(\Re{(U_\epsilon)})\nonumber\\
&\leq\lim\limits_{\epsilon\rightarrow 0}\mu+\epsilon, \text{ by }(\ref{4})\\
&=\lambda_k(\Re(A)).
\end{align*}
Again by Cauchy's interlacing theorem (Corollary III.1.5, \cite{RB}), we have $\lambda_k(\Re(U))\geq\lambda_k(\Re{(A)})$ as $U$ is a dilation of $A$. Hence $\lambda_k (\Re{(U)})=\lambda_k(\Re{(A)})$. 
\end{proof}

\begin{proof}[\textit{\textbf{Proof of Theorem \ref{main theorem 2}}}]
We will prove it considering the following two cases.

\underline{\textbf{Case I}}: Suppose $k\in\mathbb{N}$. Clearly,
\begin{align*}
\overline{\Lambda_k(A)}&\subseteq\bigcap\left\{\overline{\Lambda_k(U)}: U\text{ is a unitary } N\text{-dilation of } A \text{ to } \mathcal{H}\oplus\mathbb{C}^N\right\}.
\end{align*}
Let $\xi\notin\overline{\Lambda_k(A)}$. Then, by Theorem \ref{geomertic description for operator}, there exists $\theta\in[0,2\pi)$ such that $e^{i\theta}\xi+e^{-i\theta}\overline{\xi}>\lambda_k(e^{i\theta}A+e^{-i\theta}A^*)$. Now, by Lemma \ref{key lemma for proving main theorem 2}, there exists a unitary $N$-dilation $U$ of $A$ such that $\lambda_k(e^{i\theta}U+e^{-i\theta}U^*)=\lambda_k(e^{i\theta}A+e^{-i\theta}A^*)$. So, $e^{i\theta}\xi+e^{-i\theta}\overline{\xi}>\lambda_k(e^{i\theta}U+e^{-i\theta}U^*)$. Therefore, again by Theorem \ref{geomertic description for operator}, we have $\xi\notin\overline{\Lambda_k(U)}$. This completes the proof in this case.

\underline{\textbf{Case II}}: Suppose $k=\infty$. Let us first assume that $\Lambda_{\infty}(A)=\emptyset$. Suppose $U$ is a unitary $N$-dilation of $A$ to $\mathcal{H}\oplus\mathbb{C}^N$. If possible, let $\lambda\in\Lambda_\infty(U)$. Then there exists an $\infty$-rank projection $P\in\mathcal{B}(\mathcal{H}\oplus\mathbb{C}^N)$ such that $PUP=\lambda P$. Let $\mathcal{K}=\text{ran}P\cap\mathcal{H}$. As codimension of $\mathcal{H}$ in $\mathcal{H}\oplus\mathbb{C}^N$ is $N<\infty,\; \mathcal{K}$ is infinite dimensional. Observe,
\begin{align*}
\lambda P_{\mathcal{K}}=P_{\mathcal{K}}(\lambda P)P_{\mathcal{K}}=P_{\mathcal{K}}PUPP_{\mathcal{K}}=P_{\mathcal{K}}UP_{\mathcal{K}}=P_{\mathcal{K}}AP_{\mathcal{K}}.
\end{align*}
So, $\lambda\in\Lambda_\infty(A)$, which contradicts $\Lambda_\infty(A)=\emptyset$. Hence $\Lambda_\infty(U)=\emptyset$ and we are done.

Now, let $\Lambda_{\infty}(A)\neq\emptyset$. Clearly,
\begin{align*}
\overline{\Lambda_{\infty}(A)}\subseteq\bigcap\left\{\overline{\Lambda_{\infty}(U)}: U\text{ is a unitary } N\text{-dilation of } A \text{ to } \mathcal{H}\oplus\mathbb{C}^N\right\}.
\end{align*}
Let $\xi\notin\overline{\Lambda_{\infty}(A)}$. Then, by Theorem \ref{geomertic description for operator2}, there exists $k_o\in\mathbb{N}$ such that $\xi\notin\Omega_{k_o}(A)$. So, by Theorem \ref{geomertic description for operator}, there exists $\theta\in [0,2\pi)$ such that $e^{i\theta}\xi+e^{-i\theta}\bar{\xi}>\lambda_{k_o}(e^{i\theta}A+e^{-i\theta}A^*)$. Now, by Lemma \ref{key lemma for proving main theorem 2}, there exists a unitary $N$-dilation $U$ of $A$ such that $\lambda_{k_o}(e^{i\theta}A+e^{-i\theta}A^*)=\lambda_{k_o}(e^{i\theta}U+e^{-i\theta}U^*)$. So, $e^{i\theta}\xi+e^{-i\theta}\bar{\xi}>\lambda_{k_o}(e^{i\theta}U+e^{-i\theta}U^*)$. Therefore, by Theorem \ref{geomertic description for operator}, $\xi\notin\Omega_{k_0}(U)$. Hence, by Theorem \ref{geomertic description for operator2}, $\xi\notin\bigcap\limits_{k\geqslant 1}\Omega_k(U)=\Omega_{\infty}(U)=\overline{\Lambda_{\infty}(U)}$ as $\emptyset\neq\Lambda_{\infty}(A)\subseteq\Lambda_{\infty}(U)$. 
\end{proof}

\section{Proof of Theorem \ref{main theorem 3}}

Let $A\in\mathcal{B(\mathcal{H})}$ be a contraction. Suppose $\textbf{c}=\text{diag}(c_1, c_2,\cdots)$ be a finite rank operator. Observe that
\begin{align}\label{C-numerical range}
\overline{W_\textbf{c}(A)}\subseteq\bigcap\left\{\overline{W_{\textbf{c}\oplus \textbf{o}}(U)}: U\text{ is a unitary dilation of A}\right\}.
\end{align}

 We need a few lemmas.

 \begin{lemma}[\cite{L}]\label{theorem for C numerical range}
 	Let $\textbf{c}=\text{diag}(c_1,c_2,\cdots,c_n)$ with $c_1\geq c_2\geq\cdots\geq c_n$. Suppose $A\in M_n$ and $\Re{(A)}$ has eigenvalues $\lambda_1\geq\lambda_2\geq\cdots\geq\lambda_n$. If $\alpha=\sum\limits_{j=1}^nc_j\lambda_{n-j+1}$ and $\beta=\sum\limits_{j=1}^nc_j\lambda_{j}$ then
 	\begin{align*}
 	\Re{(W_\textbf{c}(A))}=W_\textbf{c}(\Re{(A)})=\left[\alpha,\beta\right].
 	\end{align*}
 \end{lemma}

Let $A\in\mathcal{B(\mathcal{H})}$ be self-adjoint and $k\in\mathbb{N}$. Define
\begin{align*}
\nu_k(A)=\inf\limits_{\substack{\mathcal{M}\subseteq \mathcal{H}\\ dim\mathcal{M}=k}}\max\limits_{\substack{x\in \mathcal{M} \\ \Vert x\Vert=1}}\langle Ax,x\rangle.
\end{align*}
Then $\nu_k(A)=-\lambda_k(-A)$. Note that if $A\in M_n$ is self-adjoint and $1\leq k\leq n$ then $\nu_k(A)=\lambda_{n-k+1}(A)$.

\begin{lemma}\label{theorem for C numerical range-2}
	Let $A\in\mathcal{B(\mathcal{H})}$ be self-adjoint and $\textbf{c}=\text{diag}(c_1, c_2,\cdots)$ with $c_1\geq c_2\geq\cdots$ be a finite rank operator. Suppose $\alpha=\sum\limits_{j=1}^\infty c_j\nu_j(A)$ and $\beta=\sum\limits_{j=1}^nc_j\lambda_{j}(A)$. Then
	\begin{align*}
	\overline{W_\textbf{c}(A)}=[\alpha,\beta].
	\end{align*}
\end{lemma}

\begin{proof}
	Let $\{e_1,e_2,\cdots\}$ be an orthonormal basis of $\mathcal{H}$. Suppose $\mathcal{H}_n=\text{span}\{e_1,\cdots,e_n\}$ and $A_n=P_{\mathcal{H}_n}A|_{\mathcal{H}_n}$. Then
	\begin{align*}
	\overline{W_\textbf{c}(A)}
	&=\bigcap\limits_{j=1}^\infty\overline{\bigcup\limits_{n\geq j}W_{\textbf{c}_n}(A_n)}, \text{ where }\textbf{c}_n=\text{diag}(c_1,\cdots,c_n)\\
	&=\bigcap\limits_{j=1}^\infty\overline{\bigcup\limits_{n\geq j}[\alpha_n,\beta_n]},\text{ by Lemma }\ref{theorem for C numerical range},\; \alpha_n=\sum_{j=1}^nc_j\nu_j(A_n),\; \beta_n=\sum_{j=1}^nc_j\lambda_j(A_n)\\
	&=\left[\lim\limits_{n\to\infty}\inf\alpha_n,\lim\limits_{n\to\infty}\sup\beta_n\right]\\
	&=[\alpha,\beta],\text{ by Lemma }\ref{1}.
	\end{align*}
\end{proof}

We are now ready to prove Theorem \ref{main theorem 3}.

\begin{proof}[\textbf{Proof of Theorem \ref{main theorem 3}}]
	Let $A=
	\begin{pmatrix}
	0 & 1 \\
	0 & 0 \\
	\end{pmatrix}$ and $\textbf{c}=I_2$. Then $W_\textbf{c}(A)=\{0\}$. If possible, let
	\begin{align*}
	W_\textbf{c}(A)=\bigcap\left\{\overline{W_{\textbf{c}\oplus \textbf{0}}(U)}: U \text{ is a unitary dilation of A}\right\}. 
	\end{align*}
	Then 
	\begin{align*}
	W_\textbf{c}\left(\Re (A)\right)=\bigcap\left\{\overline{W_{\textbf{c}\oplus \textbf{0}}\left(\Re (U)\right)}: U \text{ is a unitary dilation of A}\right\}.
	\end{align*}
	Now, by Lemma \ref{theorem for C numerical range-2}, we obtain $\overline{W_{\textbf{c}\oplus \textbf{0}}\left(\Re(U)\right)}=\left[\alpha_U,\beta_U\right]$ with $\alpha_U\leqslant 0\leqslant\beta_U$, where $\alpha_U=\nu_1(\Re(U))+\nu_2(\Re(U)) \text{ and } \beta_U=\lambda_1(\Re(U))+\lambda_2(\Re(U))$. So,
	\begin{align*}
	\{0\}=W_\textbf{c}(\Re (A))=\bigcap\limits_U\left[\alpha_U,\beta_U\right].
	\end{align*}
Then there exists a sequence of unitary operators $\{U_n\}_{n=1}^{\infty}$ such that $ \beta_{U_n}\rightarrow 0$ whenever $n\rightarrow \infty$. Given $\epsilon>0,$ there exists $n_{\circ}\in\mathbb{N}$ such that $\beta_{U_{n_{\circ}}}<\epsilon$, that is, $\lambda_2(\Re(U_{n_{\circ}})+\lambda_1(\Re(U_{n_{\circ}}))<\epsilon,$ which implies that $\overline{W(U_{n_{\circ}})}=\overline{\text{conv}}\sigma(U_{n_0})$ does not contain $W(A)=\{z\in\mathbb{C}: \vert z\vert\leqslant\frac{1}{2}\}$. It is a contradiction as  $U_{n_{\circ}}$ is a dilation of $A$. Hence the intersection of the closure of the $\textbf{c}\oplus \textbf{0}$-numerical ranges of all unitary dilations of $A$ contains $W_\textbf{c}(A)$ as a proper subset.

\end{proof}

\bibliographystyle{elsarticle-num}

\end{document}